\documentclass{article}
\usepackage{amssymb}
\usepackage{amsmath}
\usepackage{amsfonts}
\usepackage{cite}

\newtheorem{theorem}{Theorem}

\newtheorem{remark}[theorem]{Remark}

\newenvironment{proof}[1][Proof]{\noindent\textbf{#1.} }{\ \rule{0.5em}{0.5em}}

\begin{document}

\title{Recursive formulas for $_{2}F_{1}$ and $_{3}F_{2}$ hypergeometric
series}
\author{J. L. Gonz\'{a}lez-Santander ORCID: 0000-0001-5348-4967 \and %
C/ Ovidi Montllor i Mengual 7, pta.9 \and 46017,
Valencia, Spain. \\
%EndAName
juanluis.gonzalezsantander@gmail.com\\
}
\date{}
\maketitle

\begin{abstract}
Recursive formulas extending some known $_{2}F_{1}$ and $_{3}F_{2}$
summation formulas by using contiguous relations have been obtained. On the
one hand, these recursive equations are quite suitable for symbolic and
numerical evaluation by means of computer algebra. On the other hand,
sometimes closed-forms of such extensions can be derived by induction. It is
expected that the method used to obtain the different recursive equations
can be applied to extend other hypergeometric summation formulas given in
the literature.
\end{abstract}



\section{Introduction}

Hypergeometric and generalized hypergeometric functions have many
applications as solutions of problems concerning mathematics \cite{Jason,
Branges}\ and physics \cite{Calabra, Moch}. It is worth noting that whenever
these functions can be expressed in terms of gamma functions, the results
are very important in many applications and also from a theoretical point of
view. However, few summation theorems are available in the literature \cite%
{Andrews}. Thereby, in the last decades, many publications have devoted to
broaden these classical results. For instance, it is well-known that by
repeated application of the contiguous relations of the hypergeometric
function $_{2}F_{1}\left( a,b;c;z\right) $ \cite[Sect. 15.5(ii)]{DLMF}, any
function $_{2}F_{1}\left( a+k,b+\ell ;c+m;z\right) $, in which $k$, $\ell $,
and $m$ are integers, can be expressed as as a linear combination of $%
_{2}F_{1}\left( a,b;c;z\right) $ and any of its contiguous functions, with
coefficients that are rational functions of $a$, $b$, $c$, and $z$. Also, by
systematic exploitation of the relations between contiguous functions given
by generalized hypergeometric functions $_{p}F_{q}$ \cite[Sect. 48]%
{Rainville}, we can find in \cite{LavoieRathie} a generalization of Watson's
theorem,\ and in \cite{Lavoie} a generalization of Dixon's theorem. However,
these generalizations extend these theorems for a finite number of integer
numbers summed to certain parameters of the corresponding hypergeometric
functions.

In this paper, we provide recursive formulas in order to extend some known
summation formulas of $_{2}F_{1}$ hypergeometric function at arguments $%
z\neq 1$, and $_{3}F_{2}$ hypergeometric function at argument $z=1$. These
recursive formulas are very suitable for symbolic computation and numerical
evaluation by using computer algebra, and they are not restricted to a
finite number of cases, as in the papers stated above. Despite the fact that these
recursive formulas might be automatically computed using creative
telescoping (also called Wilf-Zeilberger's theory \cite[Sects. 3.10\&11]%
{Andrews}) following the method describe in \cite{Koomwinder}\ for
non-terminating series, we propose here a much more simple approach using contiguous
relations. Also, it is sometimes possible to derive general expressions 
in closed-form by induction from the corresponding recursive equation. 
In fact, the method described in this paper is quite general and can be used 
for many other summation formulas given in the literature, so that here we present 
some selected examples.

This article is organized as follows. In Section \ref{Section: 2F1}\ we
derive some recursive formulas for $_{2}F_{1}$ hypergeometric function at
arguments $z=\frac{1}{2},2,-1$. Section \ref{Section: 3F2}\ extends some
known $_{3}F_{2}$ hypergeometric summation formulas at argument unity, such
as Pfaff-Saalschutz sum, Watson's sum and Dixon's sum, among others.
Finally, the conclusions are collected in Section \ref{Section: Conclusions}.

\section{$_{2}F_{1}$ recursive formulas\label{Section: 2F1}}

In this Section we consider Gauss's hypergeometric series, defined as%
\begin{equation}
_{2}F_{1}\left( \left. 
\begin{array}{c}
a,b \\ 
c%
\end{array}%
\right\vert z\right) =\sum_{m=0}^{\infty }\frac{\left( a\right) _{m}\left(
b\right) _{m}}{m!\left( c\right) _{m}}z^{m},  \label{2F1_def}
\end{equation}%
where $\left( \alpha \right) _{m}=\Gamma \left( \alpha +m\right) /\Gamma
\left( \alpha \right) $ denotes the Pochhammer symbol. When $a$ or $b$ are
negative integers, the series (\ref{2F1_def})\ terminates, so that it
converges. If it is not so, (\ref{2F1_def}) is absolutely convergent when $%
\left\vert z\right\vert <1$, except when the parameter $c=0,-1,-2,\ldots $
In the case in which $\left\vert z\right\vert =1$, the series is absolutely
convergent when \textrm{Re}$\left( c-a-b\right) >0$, conditionally
convergent when $-1<\,$\textrm{Re}$\left( c-a-b\right) \leq 0$ and $z\neq 1$%
;\ and divergent when \textrm{Re}$\left( c-a-b\right) \leq -1$.

In order to obtain the recursive formulas stated below, we will use known
contiguous relations connecting Gauss`s hypergeometric series $%
_{2}F_{1}\left( a,b;c;z\right) $ with other two hypergeometric series which
vary their parameters $a$, $b$, $c$, $\pm 1$ unit. The idea is to
particularize the parameters in the contiguous relation in such a way that
we can define a recursive relation in which the initial iteration ($k=0$)\
is given by a known summation formula. Iterating the recursive relation and
with the aid of computer algebra is quite easy to obtain summation formulas
for other integers $k$.

First, we generalize Gauss's second theorem. In \cite{LavoieWhimple}\ we can
find an extension of Gauss's second theorem for $j=0$ and some particular
values of $k$.

\begin{theorem}
If $k\in 
%TCIMACRO{\U{2115} }%
%BeginExpansion
\mathbb{N}
%EndExpansion
$ and $j=0,1$, then the following recursive equation holds true:%
\begin{equation}
G_{k}\left( a,b,j\right) =G_{k-1}\left( a,b,j\right) +\frac{a}{a+b+1+j}%
G_{k-1}\left( a+1,b+1,j\right) ,  \label{Recursive_Gauss_2nd}
\end{equation}%
where%
\begin{equation*}
G_{k}\left( a,b,j\right) =\,_{2}F_{1}\left( \left. 
\begin{array}{c}
a,b+k \\ 
\frac{a+b+j+1}{2}%
\end{array}%
\right\vert \frac{1}{2}\right) ,
\end{equation*}%
and, according to Gauss's second theorem \cite[Eqn. 15.4.28]{DLMF}%
\begin{equation*}
G_{0}\left( a,b,0\right) =\frac{\sqrt{\pi }\Gamma \left( \frac{a+b+1}{2}%
\right) }{\Gamma \left( \frac{a+1}{2}\right) \Gamma \left( \frac{b+1}{2}%
\right) },\quad a+b\neq -1,-2\ldots ,
\end{equation*}%
and according to \cite[Eqn. 15.4.29]{DLMF}%
\begin{eqnarray}
&&G_{0}\left( a,b,1\right)  \label{G0(a,b,j=1)} \\
&=&\frac{2\sqrt{\pi }}{a-b}\Gamma \left( \frac{a+b}{2}+1\right) \left[ \frac{%
1}{\Gamma \left( \frac{a}{2}\right) \Gamma \left( \frac{b+1}{2}\right) }-%
\frac{1}{\Gamma \left( \frac{a+1}{2}\right) \Gamma \left( \frac{b}{2}\right) 
}\right] .  \notag
\end{eqnarray}
\end{theorem}

\begin{proof}
Consider the contiguous relation \cite[Eqn. 9.2.13]{Lebedev} 
\begin{equation*}
_{2}F_{1}\left( \left. 
\begin{array}{c}
\alpha ,\beta +1 \\ 
\gamma%
\end{array}%
\right\vert z\right) =\,_{2}F_{1}\left( \left. 
\begin{array}{c}
\alpha ,\beta \\ 
\gamma%
\end{array}%
\right\vert z\right) +\frac{\alpha z}{\gamma }\,_{2}F_{1}\left( \left. 
\begin{array}{c}
\alpha +1,\beta +1 \\ 
\gamma +1%
\end{array}%
\right\vert z\right) ,
\end{equation*}%
and set $\alpha =a$, $\beta =b+k$, $\gamma =\frac{a+b+j+1}{2}$, and $z=\frac{%
1}{2}$.
\end{proof}

\bigskip

From the above recursive equation (\ref{Recursive_Gauss_2nd}), we prove the
following identity.

\begin{theorem}
If $k\in 
%TCIMACRO{\U{2115} }%
%BeginExpansion
\mathbb{N}
%EndExpansion
$, then%
\begin{equation}
_{2}F_{1}\left( \left. 
\begin{array}{c}
a,a+k \\ 
a+1%
\end{array}%
\right\vert \frac{1}{2}\right) =2^{a}\left[ 2^{k-1}-\frac{k-1}{a+1}%
\,_{2}F_{1}\left( \left. 
\begin{array}{c}
2-k,a+1 \\ 
a+2%
\end{array}%
\right\vert -1\right) \right] ,  \label{Identity_2F1}
\end{equation}%
where notice that the hypergeometric sum on the RHS\ of (\ref{Identity_2F1}%
)\ is a finite sum.
\end{theorem}

\begin{proof}
Taking $b=a$ and $j=1$, (\ref{Recursive_Gauss_2nd})\ is reduced to%
\begin{equation}
G_{k}\left( a\right) =G_{k-1}\left( a\right) +\frac{a}{2\left( a+1\right) }%
G_{k-1}\left( a+1\right) ,  \label{Recursive_G_k(a)}
\end{equation}%
where%
\begin{equation}
G_{k}\left( a\right) =\,_{2}F_{1}\left( \left. 
\begin{array}{c}
a,a+k \\ 
a+1%
\end{array}%
\right\vert \frac{1}{2}\right) ,  \label{G_k(a)}
\end{equation}%
and according to \cite[Eqn. 7.3.7(16)]{Prudnikov3}%
\begin{equation}
G_{0}\left( a\right) =2^{a-1}a\left[ \psi \left( \frac{a+1}{2}\right) -\psi
\left( \frac{a}{2}\right) \right] ,  \label{G0(a)}
\end{equation}%
where $\psi \left( z\right) $ denotes the digamma function. Now, from (\ref%
{Recursive_G_k(a)}) and (\ref{G0(a)}), and with the aid of computer algebra,
perform the first iterations as:%
\begin{eqnarray*}
G_{1}\left( a\right) &=&2^{a}, \\
G_{2}\left( a\right) &=&2^{a}\left( 2-\frac{1}{a+1}\right) , \\
G_{3}\left( a\right) &=&2^{a}\left( 4-\frac{2}{a+1}-\frac{2}{a+2}\right) , \\
G_{4}\left( a\right) &=&2^{a}\left( 8-\frac{3}{a+1}-\frac{6}{a+2}-\frac{3}{%
a+3}\right) ,
\end{eqnarray*}%
thus, we can establish the conjecture%
\begin{equation}
G_{k}\left( a\right) =2^{a}\left( 2^{k-1}-\sum_{i=1}^{k-1}\frac{\left(
k-1\right) !}{\left( i-1\right) !\left( k-i-1\right) !\left( a+i\right) }%
\right) ,  \label{Gk(a)_sum}
\end{equation}%
which can be proved by induction. Finally, rewrite (\ref{Gk(a)_sum})\
expressing the sum\ therein as a hypergeometric function and match the
result to (\ref{G_k(a)}), to obtain (\ref{Identity_2F1}).
\end{proof}

\bigskip

Next, we generalized the result given in \cite{KimRathie}, which is given in
table form for $k=-1,-2,-3,-4,-5$.

\begin{theorem}
If $k\in 
%TCIMACRO{\U{2124} }%
%BeginExpansion
\mathbb{Z}
%EndExpansion
^{-}$ and $n\in 
%TCIMACRO{\U{2124} }%
%BeginExpansion
\mathbb{Z}
%EndExpansion
^{+}$, then the following recursive equation holds true:%
\begin{eqnarray}
&&G_{k}\left( n,a\right)  \label{Recursive_Srivastava} \\
&=&\frac{2\left( a+k-1\right) \left( 2a+k+n-1\right) }{k\left( 2a+k-1\right) 
}G_{k+1}\left( n,a\right) -\frac{2a+k-2}{k}G_{k+1}\left( n,a-1\right) , 
\notag
\end{eqnarray}%
where%
\begin{equation*}
G_{k}\left( n,a\right) =\,_{2}F_{1}\left( \left. 
\begin{array}{c}
-n,a \\ 
2a-1+k%
\end{array}%
\right\vert 2\right) ,
\end{equation*}%
and, according to \cite{Srivastava}%
\begin{eqnarray*}
&&G_{0}\left( n,a\right) \\
&=&\frac{\Gamma \left( a-\frac{1}{2}\right) }{\sqrt{\pi }}\left[ \frac{%
1+\left( -1\right) ^{n}}{2}\frac{\Gamma \left( \frac{n+1}{2}\right) }{\Gamma
\left( a+\frac{n-1}{2}\right) }-\frac{1-\left( -1\right) ^{n}}{2}\frac{%
\Gamma \left( \frac{n}{2}+1\right) }{\Gamma \left( a+\frac{n}{2}\right) }%
\right] .
\end{eqnarray*}
\end{theorem}

\begin{proof}
Consider the following contiguous relations \cite[Eqns. 9.2.7\&14]{Lebedev}%
\begin{eqnarray}
&&\gamma \left( \gamma +1\right) \,_{2}F_{1}\left( \left. 
\begin{array}{c}
\alpha ,\beta \\ 
\gamma%
\end{array}%
\right\vert z\right)  \label{Lebedev_1} \\
&=&\gamma \left( \gamma -\alpha +1\right) \,_{2}F_{1}\left( \left. 
\begin{array}{c}
\alpha ,\beta +1 \\ 
\gamma +2%
\end{array}%
\right\vert z\right)  \notag \\
&&+\alpha \left[ \gamma -\left( \gamma -\beta \right) z\right]
\,_{2}F_{1}\left( \left. 
\begin{array}{c}
\alpha +1,\beta +1 \\ 
\gamma +2%
\end{array}%
\right\vert z\right) ,  \notag
\end{eqnarray}%
and%
\begin{eqnarray}
&&_{2}F_{1}\left( \left. 
\begin{array}{c}
\alpha ,\beta +1 \\ 
\gamma +1%
\end{array}%
\right\vert z\right) -\,_{2}F_{1}\left( \left. 
\begin{array}{c}
\alpha ,\beta \\ 
\gamma%
\end{array}%
\right\vert z\right)  \label{Lebedev_2} \\
&=&\,\frac{\alpha \left( \gamma -\beta \right) z}{\gamma \left( \gamma
+1\right) }\,_{2}F_{1}\left( \left. 
\begin{array}{c}
\alpha +1,\beta +1 \\ 
\gamma +2%
\end{array}%
\right\vert z\right) .  \notag
\end{eqnarray}%
Therefore, eliminating the hypergeometric function on the RHS\ of (\ref%
{Lebedev_1})\ and (\ref{Lebedev_2}), we arrive at%
\begin{eqnarray*}
&&\gamma \,_{2}F_{1}\left( \left. 
\begin{array}{c}
\alpha ,\beta \\ 
\gamma%
\end{array}%
\right\vert z\right) -\frac{\left( \gamma -\alpha +1\right) \left( \gamma
-\beta \right) z}{\gamma +1}\,_{2}F_{1}\left( \left. 
\begin{array}{c}
\alpha ,\beta +1 \\ 
\gamma +2%
\end{array}%
\right\vert z\right) \\
&=&\left[ \gamma -\left( \gamma -\beta \right) z\right] \,_{2}F_{1}\left(
\left. 
\begin{array}{c}
\alpha ,\beta +1 \\ 
\gamma +1%
\end{array}%
\right\vert z\right) .
\end{eqnarray*}%
Substitute now $\alpha =-n$, $\beta =a-1$, $\gamma =2\left( a-1\right) +k$,
and $z=2$ to obtain (\ref{Recursive_Srivastava}).
\end{proof}

\bigskip

Finally, we generalize Kummer's theorem. This case has been discussed more
extensively in \cite{Choi}.

\begin{theorem}
If $k\in 
%TCIMACRO{\U{2115} }%
%BeginExpansion
\mathbb{N}
%EndExpansion
$, then%
\begin{equation}
G_{k+1}\left( a,b\right) =\frac{b+k}{k}G_{k}\left( a,b\right) -\frac{%
2b\left( a+k+1\right) }{k\left( a-b+1\right) }G_{k}\left( a+2,b+1\right) ,
\label{Recursive_Kummer}
\end{equation}%
where%
\begin{equation}
G_{k}\left( a,b\right) =\,_{2}F_{1}\left( \left. 
\begin{array}{c}
a+k,b \\ 
a-b+1%
\end{array}%
\right\vert -1\right) ,  \label{Gk(a,b)_Kummer}
\end{equation}%
and, according to \cite{Choi}%
\begin{eqnarray*}
&&G_{1}\left( a,b\right) \\
&=&\frac{\sqrt{\pi }\Gamma \left( a+1-b\right) }{2^{a+1}}\left( \frac{1}{%
\Gamma \left( \frac{a}{2}+1\right) \Gamma \left( \frac{a+1}{2}-b\right) }+%
\frac{1}{\Gamma \left( \frac{a+1}{2}\right) \Gamma \left( \frac{a}{2}%
-b+1\right) }\right) .
\end{eqnarray*}
\end{theorem}

\begin{proof}
From the identity \cite[Eqn. 15.5.20]{DLMF}\ and the differentiation formula 
\cite[Eqn. 15.5.1]{DLMF}, setting $\alpha \rightarrow \alpha +1$, we arrive
at%
\begin{eqnarray*}
&&z\left( 1-z\right) \frac{\left( \alpha +1\right) \beta }{\gamma }%
\,_{2}F_{1}\left( \left. 
\begin{array}{c}
\alpha +2,\beta +1 \\ 
\gamma +1%
\end{array}%
\right\vert z\right) \\
&=&\left( \gamma -\alpha +1\right) \,_{2}F_{1}\left( \left. 
\begin{array}{c}
\alpha ,\beta \\ 
\gamma%
\end{array}%
\right\vert z\right) +\left( \alpha +1-\gamma +\beta z\right)
\,_{2}F_{1}\left( \left. 
\begin{array}{c}
\alpha +1,\beta \\ 
\gamma%
\end{array}%
\right\vert z\right) .
\end{eqnarray*}%
Thereby, taking $\alpha =a+k$, $\beta =b$, $\gamma =a-b+1$, and $z=-1$, we
arrive at (\ref{Recursive_Kummer}).
\end{proof}

\bigskip

Note that $\forall k=0$, (\ref{Gk(a,b)_Kummer})\ reduces to Kummer's
summation formula \cite[Eqn. 15.4.26]{DLMF}, but in this case, the recursive
equation (\ref{Recursive_Kummer}) collapses. However, it is worth noting
that in \cite{Choi}\ we find a closed-form of (\ref{Gk(a,b)_Kummer}), which
reads as%
\begin{equation}
G_{k}\left( a,b\right) =\frac{\Gamma \left( 1+a-b\right) }{2\Gamma \left(
a+k\right) }\sum_{m=0}^{k}\binom{k}{m}\frac{\Gamma \left( \frac{a+k+m}{2}%
\right) }{\Gamma \left( \frac{a-k+m}{2}-b+1\right) }.  \label{Gk_Choi}
\end{equation}

As by-product, we can obtain an interesting identity, inserting (\ref%
{Gk_Choi}) in the recursive formula (\ref{Recursive_Kummer}). Direct
substitution yields%
\begin{eqnarray}
&&\sum_{m=0}^{k+1}\binom{k+1}{m}\frac{\Gamma \left( \frac{a+k+m+1}{2}\right) 
}{\Gamma \left( \frac{a-k+m+1}{2}-b\right) }  \label{LHS_Choi} \\
&=&\sum_{m=0}^{k}\binom{k}{m}\left( a+k-\frac{bm}{k}\right) \frac{\Gamma
\left( \frac{a+k+m}{2}\right) }{\Gamma \left( \frac{a-k+m}{2}-b+1\right) }. 
\notag
\end{eqnarray}

Now, recast (\ref{LHS_Choi})\ as 
\begin{equation*}
\frac{\Gamma \left( \frac{a+k+1}{2}\right) }{\Gamma \left( \frac{a-k+1}{2}%
-b\right) }+\sum_{m=0}^{k}\binom{k}{m}\frac{\left( k+1\right) \left(
a+k+m\right) }{2\left( m+1\right) }\frac{\Gamma \left( \frac{a+k+m}{2}%
\right) }{\Gamma \left( \frac{a-k+m}{2}-b+1\right) },
\end{equation*}%
thus we obtain this interesting formula, $\forall k=1,2,\ldots $%
\begin{eqnarray}
&&\frac{\Gamma \left( \frac{a+k+1}{2}\right) }{\Gamma \left( \frac{a-k+1}{2}%
-b\right) }  \label{Identity_Choi} \\
&=&\sum_{m=0}^{k}\binom{k}{m}\left( a+\frac{k-1}{2}-\frac{bm}{k}-\frac{%
\left( k+1\right) \left( a+k-1\right) }{2\left( m+1\right) }\right) \frac{%
\Gamma \left( \frac{a+k+m}{2}\right) }{\Gamma \left( \frac{a-k+m}{2}%
-b+1\right) }.  \notag
\end{eqnarray}

It is worth noting that (\ref{Identity_Choi}) can be proven using computer
algebra.

\section{$_{3}F_{2}$ recursive formulas\label{Section: 3F2}}

The generalized hypergeometric series $_{p}F_{q}$ is a natural
generalization of the Gauss's series $_{2}F_{1}$, and is defined as 
\begin{equation}
_{p}F_{q}\left( \left. 
\begin{array}{c}
a_{1},\ldots ,a_{p} \\ 
b_{1},\ldots ,b_{q}%
\end{array}%
\right\vert z\right) =\sum_{m=0}^{\infty }\frac{\left( a_{1}\right)
_{m}\cdots \left( a_{p}\right) _{m}}{m!\left( b_{1}\right) _{m}\cdots \left(
b_{q}\right) _{m}}z^{m}.  \label{pFq_def}
\end{equation}

If any $a_{j}$, $j=1\ldots p$ is a negative integer, then the series (\ref%
{pFq_def})\ terminates, thus it converges. If (\ref{pFq_def})\ is not a
terminating series, then it converges $\forall \left\vert z\right\vert
<\infty $, if $p\leq q$; and $\forall \left\vert z\right\vert <1$, if $p=q+1$%
. Also, (\ref{pFq_def})\ diverges $\forall z\neq 0$, if $p>q+1$. If $p=q+1$
and $\left\vert z\right\vert =1$, then the series (\ref{pFq_def})\ is
absolutely convergent when \textrm{Re}$\left(
\sum_{j=1}^{q}b_{j}-\sum_{j=1}^{p}a_{j}\right) >0$; conditionally convergent
when $z\neq 1$ and $-1<\,$\textrm{Re}$\left(
\sum_{j=1}^{q}b_{j}-\sum_{j=1}^{p}a_{j}\right) <0$; and divergent when $\,$%
\textrm{Re}$\left( \sum_{j=1}^{q}b_{j}-\sum_{j=1}^{p}a_{j}\right) \leq -1$.

\begin{theorem}
If $k\in 
%TCIMACRO{\U{2115} }%
%BeginExpansion
\mathbb{N}
%EndExpansion
$, then 
\begin{eqnarray}
G_{k}\left( a,b,c,d\right) &=&G_{k-1}\left( a,b,c,d\right)
\label{Recursive_Miller} \\
&&+\frac{ab}{c\left( d+1\right) }G_{k-1}\left( a+1,b+1,c+1,d+1\right) , 
\notag
\end{eqnarray}%
where%
\begin{equation}
G_{k}\left( a,b,c,d\right) =\,_{3}F_{2}\left( \left. 
\begin{array}{c}
a,b,c+k+1 \\ 
d+1,c%
\end{array}%
\right\vert 1\right) ,  \label{G_k_Miller_def}
\end{equation}%
and where, renaming parameters in \cite{Miller}, we have%
\begin{equation}
G_{0}\left( a,b,c,d\right) =\frac{\Gamma \left( d+1\right) \Gamma \left(
d-a-b\right) }{c\,\Gamma \left( d-a+1\right) \Gamma \left( d-b+1\right) }%
\left[ a\left( b-c\right) +c\left( d-b\right) \right] .  \label{G0_Miller}
\end{equation}
\end{theorem}

\begin{proof}
Consider the contiguous relation \cite[Eqn. 3.7.9]{Andrews}, exchanging the
parameters $\alpha \longleftrightarrow \gamma $, 
\begin{eqnarray}
&&_{3}F_{2}\left( \left. 
\begin{array}{c}
\alpha ,\beta ,\gamma +1 \\ 
\delta ,\varepsilon%
\end{array}%
\right\vert 1\right)  \label{Miller_1} \\
&=&\,_{3}F_{2}\left( \left. 
\begin{array}{c}
\alpha ,\beta ,\gamma \\ 
\delta ,\varepsilon%
\end{array}%
\right\vert 1\right) +\frac{\alpha \beta }{\delta \varepsilon }%
\,_{3}F_{2}\left( \left. 
\begin{array}{c}
\alpha +1,\beta +1,\gamma +1 \\ 
\delta +1,\varepsilon +1%
\end{array}%
\right\vert 1\right) .  \notag
\end{eqnarray}%
and perform the substitutions $\alpha =a$, $\beta =b$, $\gamma =c+k+1$, $%
\delta =d+1$ and $\varepsilon =c$, to obtain the recursive formula (\ref%
{Recursive_Miller}).
\end{proof}

\begin{remark}
Notice that $G_{0}\left( a,b,c,c\right) $ reduces to Gauss's summation
formula \cite[Eqn. 15.4.20]{DLMF}.
\end{remark}

\bigskip

From the above recursive equation (\ref{Recursive_Miller}), we obtain the
following identity.

\begin{theorem}
If $k=0,1,2\ldots $, then%
\begin{eqnarray}
&&\,_{3}F_{2}\left( \left. 
\begin{array}{c}
a,b,c+k+1 \\ 
d+1,c%
\end{array}%
\right\vert 1\right)  \label{Miller_relation} \\
&=&\frac{\left( -1\right) ^{k}\Gamma \left( d+1\right) \Gamma \left(
d-a-b-k\right) \left( a+b-d+1\right) _{k}}{c\,\Gamma \left( d-a+1\right)
\Gamma \left( d-b+1\right) }  \notag \\
&&\times \left\{ \left[ a\left( b-c\right) +c\left( d-b\right) \right]
\,_{3}F_{2}\left( \left. 
\begin{array}{c}
-k,a,b \\ 
c+1,a+b-d+1%
\end{array}%
\right\vert 1\right) \right.  \notag \\
&&+\left. \frac{ab\left( c-d\right) k}{\left( c+1\right) \left(
a+b-d+1\right) }\,_{3}F_{2}\left( \left. 
\begin{array}{c}
1-k,a+1,b+1 \\ 
c+2,a+b-d+2%
\end{array}%
\right\vert 1\right) \right\} ,  \notag
\end{eqnarray}%
where notice that on the RHS\ of (\ref{Miller_relation})\ the hypergeometric
sums are finite sums.
\end{theorem}

\begin{proof}
With the aid of computer algebra, we can iterate the first terms of the
recursive equation (\ref{Recursive_Miller}), starting from (\ref{G0_Miller}%
), obtaining:%
\begin{eqnarray*}
&&G_{1}\left( a,b,c,d\right) \\
&=&-\frac{\Gamma \left( d+1\right) \Gamma \left( d-a-b-1\right) }{c\,\Gamma
\left( d-a+1\right) \Gamma \left( d-b+1\right) } \\
&&\left\{ \left( a+b-d-1\right) \left[ a\left( b-c\right) +c\left(
d-b\right) \right] 
\begin{array}{c}
%TCIMACRO{\TeXButton{TeX field}{\displaystyle} }%
%BeginExpansion
\displaystyle
%EndExpansion
\\ 
%TCIMACRO{\TeXButton{TeX field}{\displaystyle}}%
%BeginExpansion
\displaystyle%
%EndExpansion
\end{array}%
\right. \\
&&\left. 
\begin{array}{c}
%TCIMACRO{\TeXButton{TeX field}{\displaystyle} }%
%BeginExpansion
\displaystyle
%EndExpansion
\\ 
%TCIMACRO{\TeXButton{TeX field}{\displaystyle}}%
%BeginExpansion
\displaystyle%
%EndExpansion
\end{array}%
-\frac{ab\left[ a\left( b-c\right) -\left( b+1\right) c+\left( c+1\right) d%
\right] }{c+1}\right\} , \\
&&G_{2}\left( a,b,c,d\right) \\
&=&\frac{\Gamma \left( d+1\right) \Gamma \left( d-a-b-2\right) }{c\,\Gamma
\left( d-a+1\right) \Gamma \left( d-b+1\right) } \\
&&\left\{ \left( a+b-d-1\right) \left( a+b-d-2\right) \left[ a\left(
b-c\right) +c\left( d-b\right) \right] 
\begin{array}{c}
%TCIMACRO{\TeXButton{TeX field}{\displaystyle} }%
%BeginExpansion
\displaystyle
%EndExpansion
\\ 
%TCIMACRO{\TeXButton{TeX field}{\displaystyle}}%
%BeginExpansion
\displaystyle%
%EndExpansion
\end{array}%
\right. \\
&&-\frac{2ab\left( a+b-d-2\right) \left[ a\left( b-c\right) -\left(
b+1\right) c+\left( c+1\right) d\right] }{c+1} \\
&&\left. 
\begin{array}{c}
%TCIMACRO{\TeXButton{TeX field}{\displaystyle} }%
%BeginExpansion
\displaystyle
%EndExpansion
\\ 
%TCIMACRO{\TeXButton{TeX field}{\displaystyle}}%
%BeginExpansion
\displaystyle%
%EndExpansion
\end{array}%
+\frac{a\left( a+1\right) b\left( b+1\right) \left[ a\left( b-c\right)
-\left( b+2\right) c+\left( c+2\right) d\right] }{\left( c+1\right) \left(
c+2\right) }\right\} .
\end{eqnarray*}%
Therefore, we can conjecture the following general form,%
\begin{eqnarray}
&&G_{k}\left( a,b,c,d\right)  \label{G_k_Sum} \\
&=&\frac{\left( -1\right) ^{k}\Gamma \left( d+1\right) \Gamma \left(
d-a-b-k\right) }{c\,\Gamma \left( d-a+1\right) \Gamma \left( d-b+1\right) } 
\notag \\
&&\sum_{j=0}^{k}\frac{\left( -1\right) ^{j}\left( a+b-d+j+1\right)
_{k-j}\left( a\right) _{j}\left( b\right) _{j}}{\left( c+1\right) _{j}}%
\binom{k}{j}  \notag \\
&&\quad \times \left[ a\left( b-c\right) -\left( b+j\right) c+\left(
c+j\right) d\right] ,  \notag
\end{eqnarray}%
which can be proved by induction. Finally, split the sum given in (\ref%
{G_k_Sum})\ in two sums and recast them as hypergeometric sums, to obtain,%
\begin{eqnarray}
&&G_{k}\left( a,b,c,d\right)  \label{G_k_resultado} \\
&=&\frac{\left( -1\right) ^{k}\Gamma \left( d+1\right) \Gamma \left(
d-a-b-k\right) \Gamma \left( a+b-d+k+1\right) }{c\,\Gamma \left(
d-a+1\right) \Gamma \left( d-b+1\right) \Gamma \left( a+b-d+1\right) } 
\notag \\
&&\left\{ \left[ a\left( b-c\right) +c\left( d-c\right) \right]
\,_{3}F_{2}\left( \left. 
\begin{array}{c}
-k,a,b \\ 
c+1,a+b-d+1%
\end{array}%
\right\vert 1\right) \right.  \notag \\
&&+\left. \frac{ab\left( c-d\right) k}{\left( c+1\right) \left(
a+b-d+1\right) }\,_{3}F_{2}\left( \left. 
\begin{array}{c}
1-k,a+1,b+1 \\ 
c+2,a+b-d+2%
\end{array}%
\right\vert 1\right) \right\} .  \notag
\end{eqnarray}%
Finally, match (\ref{G_k_resultado})\ to (\ref{G_k_Miller_def}),\ to obtain (%
\ref{Miller_relation}).
\end{proof}

\bigskip

Next, we consider an extension of Pfaff-Saalschutz summation formula.

\begin{theorem}
If $k,n\in 
%TCIMACRO{\U{2115} }%
%BeginExpansion
\mathbb{N}
%EndExpansion
$, then 
\begin{eqnarray}
&&G_{k}\left( n,a,b,c\right) =G_{k-1}\left( n,a+1,b,c+1\right)
\label{Recursive_Pfaff} \\
&&+\frac{nb}{\left( c+k\right) \left( a+b+1-n-c\right) }G_{k-1}\left(
n-1,a+1,b+1,c+2\right) ,  \notag
\end{eqnarray}%
where%
\begin{equation*}
G_{k}\left( n,a,b,c\right) =\,_{3}F_{2}\left( \left. 
\begin{array}{c}
-n,a,b \\ 
c+k,a+b+1-n-c%
\end{array}%
\right\vert 1\right) ,
\end{equation*}%
and, according to Pfaff-Saalschutz balanced sum \cite[Eqn. 16.4.3]{DLMF},%
\begin{equation}
G_{0}\left( n,a,b,c\right) =\frac{\left( c-a\right) _{n}\left( c-b\right)
_{n}}{\left( c\right) _{n}\left( c-a-b\right) _{n}}.  \label{Pfaff_sum}
\end{equation}
\end{theorem}

\begin{proof}
Consider again the contiguous relation given in \cite[Eqn. 3.7.9]{Andrews},
exchanging the parameters $\alpha \longleftrightarrow \beta $, thus 
\begin{eqnarray*}
&&_{3}F_{2}\left( \left. 
\begin{array}{c}
\alpha ,\beta +1,\gamma \\ 
\delta ,\varepsilon%
\end{array}%
\right\vert 1\right) \\
&=&\,_{3}F_{2}\left( \left. 
\begin{array}{c}
\alpha ,\beta ,\gamma \\ 
\delta ,\varepsilon%
\end{array}%
\right\vert 1\right) +\frac{\alpha \gamma }{\delta \varepsilon }%
\,_{3}F_{2}\left( \left. 
\begin{array}{c}
\alpha +1,\beta +1,\gamma +1 \\ 
\delta +1,\varepsilon +1%
\end{array}%
\right\vert 1\right) .
\end{eqnarray*}%
Now, performing the substitutions $\alpha =-n$, $\beta =a$, $\gamma =b$, $%
\delta =c+k$, and $\varepsilon =a+b+1-n-c$, we arrive at the recursive
equation (\ref{Recursive_Pfaff}).
\end{proof}

\bigskip

Again, iterating the recursive equation (\ref{Recursive_Pfaff}), we can
generalize Pfaff-Saalschutz summation formula as follows:

\begin{theorem}
If $n\in 
%TCIMACRO{\U{2115} }%
%BeginExpansion
\mathbb{N}
%EndExpansion
$ and $k=0,1,2,\ldots $, then%
\begin{eqnarray}
&&_{3}F_{2}\left( \left. 
\begin{array}{c}
-n,a,b \\ 
c+k,a+b+1-n-c%
\end{array}%
\right\vert 1\right)  \label{Identity_Pfaff} \\
&=&\frac{\left( c-a\right) _{n}\left( c-b+k\right) _{n}}{\left( c+k\right)
_{n}\left( c-a-b\right) _{n}}\,_{3}F_{2}\left( \left. 
\begin{array}{c}
-k,-n,b \\ 
c-a,b-c-k-n+1%
\end{array}%
\right\vert 1\right) ,  \notag
\end{eqnarray}%
where $k=0$ matches Pfaff-Saalschutz summation formula (\ref{Pfaff_sum}).
\end{theorem}

\begin{proof}
Computing the first iterations of (\ref{Recursive_Pfaff}), starting from (%
\ref{Pfaff_sum}), we can conjecture that%
\begin{eqnarray}
&&G_{k}\left( n,a,b,c\right)  \notag \\
&=&\frac{\Gamma \left( n-a-c\right) }{\Gamma \left( c-b+k\right) \left(
c+k\right) _{n}\left( c-a-b\right) _{n}}  \notag \\
&&\sum_{j=0}^{k}\left( -1\right) ^{j}\binom{k}{j}\frac{\left( b\right)
_{j}\left( n-j+1\right) _{j}\Gamma \left( c-b+n+k-j\right) }{\Gamma \left(
c-a+j\right) },  \label{G_K_Pfaff_sum}
\end{eqnarray}%
which can be proven by induction. Finally, the sum given in (\ref%
{G_K_Pfaff_sum})\ can be rewritten as a terminating hypergeometric sum,
obtaining (\ref{Identity_Pfaff}).
\end{proof}

\begin{theorem}
If $k\in 
%TCIMACRO{\U{2115} }%
%BeginExpansion
\mathbb{N}
%EndExpansion
$, we have the recursive formula%
\begin{eqnarray}
&&G_{k+1}\left( a,b,c\right) =\frac{a\left[ a-1-2\left( b+c+k\right) \right] 
}{k\left( b+c+k\right) }G_{k}\left( a+1,b+1,c+1\right)
\label{Recursive_Dixon} \\
&&-\frac{\left( a-b\right) \left( a-c\right) }{k\left( b+c+k\right) }%
G_{k}\left( a-1,b,c\right) ,  \notag
\end{eqnarray}%
where%
\begin{equation*}
G_{k}\left( a,b,c\right) =\,_{3}F_{2}\left( \left. 
\begin{array}{c}
a,b+k,c+k \\ 
a-b+1,a-c+1%
\end{array}%
\right\vert 1\right) ,
\end{equation*}%
and where, taking $i=1$ and $j=0$ in \cite{Lavoie}, we have%
\begin{eqnarray*}
G_{1}\left( a,b,c\right) &=&\frac{\Gamma \left( 1+a-b\right) \Gamma \left(
1+a-c\right) }{2^{2c+1}bc\Gamma \left( a-2c\right) \Gamma \left(
a-b-c\right) } \\
&&\left[ \frac{\Gamma \left( \frac{a+1}{2}-c\right) \Gamma \left( \frac{a}{2}%
-b-c\right) }{\Gamma \left( \frac{a+1}{2}\right) \Gamma \left( \frac{a}{2}%
-b\right) }-\frac{\Gamma \left( \frac{a+1}{2}-b-c\right) \Gamma \left( \frac{%
a}{2}-c\right) }{\Gamma \left( \frac{a}{2}\right) \Gamma \left( \frac{a+1}{2}%
-b\right) }\right] .
\end{eqnarray*}
\end{theorem}

\begin{proof}
Consider the contiguous relation \cite[Eqn. 3.7.12]{Andrews}, performing the
substitution $\alpha \rightarrow \alpha +1$,%
\begin{eqnarray}
&&\delta \varepsilon \,_{3}F_{2}\left( \left. 
\begin{array}{c}
\alpha +1,\beta ,\gamma \\ 
\delta ,\varepsilon%
\end{array}%
\right\vert 1\right)  \label{Dixon_1} \\
&=&\left( \alpha +1\right) \left( \delta +\varepsilon -\alpha -\beta -\gamma
-2\right) \,_{3}F_{2}\left( \left. 
\begin{array}{c}
\alpha +2,\beta +1,\gamma +1 \\ 
\delta +1,\varepsilon +1%
\end{array}%
\right\vert 1\right)  \notag \\
&&+\left( \delta -\alpha -1\right) \left( \varepsilon -\alpha -1\right)
\,_{3}F_{2}\left( \left. 
\begin{array}{c}
\alpha +1,\beta +1,\gamma +1 \\ 
\delta +1,\varepsilon +1%
\end{array}%
\right\vert 1\right) ,  \notag
\end{eqnarray}%
and also the contiguous relation \cite[Eqn. 3.7.9]{Andrews} 
\begin{eqnarray}
_{3}F_{2}\left( \left. 
\begin{array}{c}
\alpha +1,\beta ,\gamma \\ 
\delta ,\varepsilon%
\end{array}%
\right\vert 1\right) &=&\,_{3}F_{2}\left( \left. 
\begin{array}{c}
\alpha ,\beta ,\gamma \\ 
\delta ,\varepsilon%
\end{array}%
\right\vert 1\right)  \label{Dixon_2} \\
&&+\frac{\beta \gamma }{\delta \varepsilon }\,_{3}F_{2}\left( \left. 
\begin{array}{c}
\alpha +1,\beta +1,\gamma +1 \\ 
\delta +1,\varepsilon +1%
\end{array}%
\right\vert 1\right) .  \notag
\end{eqnarray}%
Thereby, eliminating the hypergeometric sum of the RHS\ of (\ref{Dixon_1})\
and (\ref{Dixon_2}), we arrive at%
\begin{eqnarray}
&&\delta \varepsilon \,_{3}F_{2}\left( \left. 
\begin{array}{c}
\alpha ,\beta ,\gamma \\ 
\delta ,\varepsilon%
\end{array}%
\right\vert 1\right)  \label{Dixon_3} \\
&&-\left( \alpha +1\right) \left( \delta +\varepsilon -\alpha -\beta -\gamma
-2\right) \,_{3}F_{2}\left( \left. 
\begin{array}{c}
\alpha +2,\beta +1,\gamma +1 \\ 
\delta +1,\varepsilon +1%
\end{array}%
\right\vert 1\right)  \notag \\
&=&\left[ \left( \delta -\alpha -1\right) \left( \varepsilon -\alpha
-1\right) -\beta \gamma \right] \,_{3}F_{2}\left( \left. 
\begin{array}{c}
\alpha +1,\beta +1,\gamma +1 \\ 
\delta +1,\varepsilon +1%
\end{array}%
\right\vert 1\right) .  \notag
\end{eqnarray}%
Finally, substituting in (\ref{Dixon_3}) $\alpha =a$, $\beta =b+k$, $\gamma
=c+k$, $\delta =a-b+1$ and $\varepsilon =a-c+1$, we arrive at the recursive
equation (\ref{Recursive_Dixon}).
\end{proof}

\begin{remark}
Notice that $\forall k=0$, $G_{0}\left( a,b,c\right) $ reduces to Dixon's
theorem \cite[Eqn. 16.4.4]{DLMF}, but in this case the recursive equation (%
\ref{Recursive_Dixon}) collapses.
\end{remark}

\bigskip

The next two recursive formulas extend Watson's sum \cite[Eqn. 16.4.6]{DLMF}%
. In \cite{LavoieWatson},\ we can find two other extensions of Watson's sum
by using contiguous relations. Next, we extend one of the results given in 
\cite{LavoieWatson}, being the latter the particular case $k=1$ of the next
theorem.

\begin{theorem}
If $k\in 
%TCIMACRO{\U{2115} }%
%BeginExpansion
\mathbb{N}
%EndExpansion
$, then the following recursive equation holds true:%
\begin{eqnarray}
&&G_{k}\left( a,b,c\right)  \label{Recursive_Watson_Lavoie} \\
&=&G_{k-1}\left( a,b,c\right) -\frac{2abc\ G_{k-1}\left( a+1,b+1,c+1\right) 
}{\left( a+b+1\right) \left( 2c+k-1\right) \left( 2c+k\right) },  \notag
\end{eqnarray}%
where%
\begin{equation*}
G_{k}\left( a,b,c\right) =\,_{3}F_{2}\left( \left. 
\begin{array}{c}
a,b,c \\ 
\frac{a+b+1}{2},2c+k%
\end{array}%
\right\vert 1\right) ,
\end{equation*}%
and, according to Watson's sum \cite[Eqn. 16.4.6]{DLMF}, we have%
\begin{equation}
G_{0}\left( a,b,c\right) =\frac{\sqrt{\pi }\Gamma \left( c+\frac{1}{2}%
\right) \Gamma \left( \frac{a+b+1}{2}\right) \Gamma \left( \frac{1-a-b}{2}%
+c\right) }{\Gamma \left( \frac{a+1}{2}\right) \Gamma \left( \frac{b+1}{2}%
\right) \Gamma \left( \frac{1-a}{2}+c\right) \Gamma \left( \frac{1-b}{2}%
+c\right) }.  \label{Watson_sum}
\end{equation}
\end{theorem}

\begin{proof}
Exchanging $\beta \longleftrightarrow \gamma $ and $\delta
\longleftrightarrow \varepsilon $ in the recursive relation given in \cite[%
Sect. 3.7]{Andrews}, we have%
\begin{eqnarray*}
&&_{3}F_{2}\left( \left. 
\begin{array}{c}
\alpha ,\beta ,\gamma +1 \\ 
\delta ,\varepsilon +1%
\end{array}%
\right\vert 1\right) \\
&=&\,_{3}F_{2}\left( \left. 
\begin{array}{c}
\alpha ,\beta ,\gamma \\ 
\delta ,\varepsilon%
\end{array}%
\right\vert 1\right) +\frac{\alpha \beta \left( \varepsilon -\gamma \right) 
}{\delta \varepsilon \left( \varepsilon +1\right) }\,_{3}F_{2}\left( \left. 
\begin{array}{c}
\alpha +1,\beta +1,\gamma +1 \\ 
\delta +1,\varepsilon +2%
\end{array}%
\right\vert 1\right) .
\end{eqnarray*}%
Also, performing the change $\varepsilon \rightarrow \varepsilon +1$ in (\ref%
{Miller_1}), 
\begin{eqnarray*}
&&_{3}F_{2}\left( \left. 
\begin{array}{c}
\alpha ,\beta ,\gamma +1 \\ 
\delta ,\varepsilon +1%
\end{array}%
\right\vert 1\right) \\
&=&\,_{3}F_{2}\left( \left. 
\begin{array}{c}
\alpha ,\beta ,\gamma \\ 
\delta ,\varepsilon +1%
\end{array}%
\right\vert 1\right) +\frac{\alpha \beta }{\delta \left( \varepsilon
+1\right) }\,_{3}F_{2}\left( \left. 
\begin{array}{c}
\alpha +1,\beta +1,\gamma +1 \\ 
\delta +1,\varepsilon +2%
\end{array}%
\right\vert 1\right) .
\end{eqnarray*}%
Therefore, equating the above equations, we arrive at 
\begin{eqnarray*}
&&_{3}F_{2}\left( \left. 
\begin{array}{c}
\alpha ,\beta ,\gamma \\ 
\delta ,\varepsilon +1%
\end{array}%
\right\vert 1\right) \\
&=&\,_{3}F_{2}\left( \left. 
\begin{array}{c}
\alpha ,\beta ,\gamma \\ 
\delta ,\varepsilon%
\end{array}%
\right\vert 1\right) -\frac{\alpha \beta \gamma }{\delta \left( \varepsilon
+1\right) }\,_{3}F_{2}\left( \left. 
\begin{array}{c}
\alpha +1,\beta +1,\gamma +1 \\ 
\delta +1,\varepsilon +2%
\end{array}%
\right\vert 1\right) .
\end{eqnarray*}%
Now, substituting $\alpha =a$, $\beta =b$, $\gamma =c$, $\delta =\frac{a+b+1%
}{2}$, and $\varepsilon =2c+k$, we arrive at the recursive equation (\ref%
{Recursive_Watson_Lavoie}).
\end{proof}

\begin{theorem}
If $k\in 
%TCIMACRO{\U{2115} }%
%BeginExpansion
\mathbb{N}
%EndExpansion
$, then we have the following recursive equation:%
\begin{eqnarray}
&&G_{k}\left( a,b,c\right) =G_{k-1}\left( a,b,c\right) +\frac{b}{a+b+1}
\label{Recursive_Watson} \\
&&\left[ \frac{\left( a+k\right) \left( b+1\right) }{\left( 2c+1\right)
\left( a+b+3\right) }G_{k-1}\left( a+2,b+2,c+1\right) +G_{k-1}\left(
a+1,b+1,c\right) \right] ,  \notag
\end{eqnarray}%
where%
\begin{equation*}
G_{k}\left( a,b,c\right) =\,_{3}F_{2}\left( \left. 
\begin{array}{c}
a+k,b,c \\ 
\frac{a+b+1}{2},2c%
\end{array}%
\right\vert 1\right) ,
\end{equation*}%
and $G_{0}\left( a,b,c\right) $ is given by Watson's sum (\ref{Watson_sum}).
\end{theorem}

\begin{proof}
Exchanging $\alpha \longleftrightarrow \beta $ and setting $\beta
\rightarrow \beta +1$, $\gamma \rightarrow \gamma +1$, and $\varepsilon
\rightarrow \varepsilon +1$ in the recursive relation given in \cite[Sect.
3.7]{Andrews}, we have%
\begin{eqnarray}
&&_{3}F_{2}\left( \left. 
\begin{array}{c}
\alpha +1,\beta +1,\gamma +1 \\ 
\delta +1,\varepsilon +1%
\end{array}%
\right\vert 1\right)  \label{Watson_1} \\
&=&\frac{\left( \beta +1\right) \left( \delta -\alpha \right) \left( \gamma
+1\right) }{\delta \left( \delta +1\right) \left( \varepsilon +1\right) }%
\,_{3}F_{2}\left( \left. 
\begin{array}{c}
\alpha +1,\beta +2,\gamma +2 \\ 
\delta +2,\varepsilon +2%
\end{array}%
\right\vert 1\right)  \notag \\
&&+\,_{3}F_{2}\left( \left. 
\begin{array}{c}
\alpha ,\beta +1,\gamma +1 \\ 
\delta ,\varepsilon +1%
\end{array}%
\right\vert 1\right) .  \notag
\end{eqnarray}%
Now, substituting the RHS\ of (\ref{Watson_1})\ in (\ref{Miller_1}), we
arrive at%
\begin{eqnarray*}
&&_{3}F_{2}\left( \left. 
\begin{array}{c}
\alpha ,\beta +1,\gamma \\ 
\delta ,\varepsilon%
\end{array}%
\right\vert 1\right) \,=\,_{3}F_{2}\left( \left. 
\begin{array}{c}
\alpha ,\beta ,\gamma \\ 
\delta ,\varepsilon%
\end{array}%
\right\vert 1\right) \\
&&\quad +\frac{\delta \varepsilon }{\alpha \gamma }\left[ \frac{\left( \beta
+1\right) \left( \delta -\alpha \right) \left( \gamma +1\right) }{\delta
\left( \delta +1\right) \left( \varepsilon +1\right) }\,_{3}F_{2}\left(
\left. 
\begin{array}{c}
\alpha +1,\beta +2,\gamma +2 \\ 
\delta +2,\varepsilon +2%
\end{array}%
\right\vert 1\right) \right. \\
&&\quad +\left. \,_{3}F_{2}\left( \left. 
\begin{array}{c}
\alpha ,\beta +1,\gamma +1 \\ 
\delta ,\varepsilon +1%
\end{array}%
\right\vert 1\right) \right] \,.
\end{eqnarray*}%
Finally, taking $\alpha =c$, $\beta =a+k$, $\gamma =b$, $\delta =2c$, and $%
\varepsilon =\frac{a+b+1}{2}$, we arrive at the recursive equation (\ref%
{Recursive_Watson}).
\end{proof}

\bigskip

Finally, we provide and extension of a summation formula given by Bailey. By
induction from the recursive formula found, a simple closed-form expression
is derived.

\begin{theorem}
If $k=0,1,2,\ldots $, then the following recursive equation is satisfied:%
\begin{eqnarray}
&&G_{k+1}\left( a,b,c\right) =\frac{\left( 2c-b+k\right) \left(
2c-b+k+1\right) }{\left( a-1\right) \left( 2c-2b+k+1\right) \left(
c-b+k\right) }  \label{Recursive_Bailey} \\
&&\left[ \left( c-1\right) G_{k}\left( a-1,b-1,c-1\right) -\left( c-a\right)
G_{k}\left( a-1,b,c\right) \right] ,  \notag
\end{eqnarray}%
where%
\begin{eqnarray}
&&G_{k}\left( a,b,c\right)  \notag \\
&=&\,_{3}F_{2}\left( \left. 
\begin{array}{c}
a,b,c+1 \\ 
1+2c-b+k,c%
\end{array}%
\right\vert 1\right)  \notag \\
&=&\frac{\left[ \left( a-2c\right) \left( b-c\right) +kc\right] \Gamma
\left( 2c-b+k+1\right) \Gamma \left( 2c-a-2b+k\right) }{c\ \Gamma \left(
2c-2b+k+1\right) \Gamma \left( 2c-a-b+k+1\right) }.
\label{Gk_Bailey_resultado}
\end{eqnarray}
\end{theorem}

\begin{proof}
Considering the contiguous relation \cite[Eqn. 3.7.14]{Andrews}, we have%
\begin{eqnarray*}
&&\varepsilon \,_{3}F_{2}\left( \left. 
\begin{array}{c}
\alpha ,\beta ,\gamma \\ 
\delta ,\varepsilon%
\end{array}%
\right\vert 1\right) -\left( \varepsilon -\alpha \right) \,_{3}F_{2}\left(
\left. 
\begin{array}{c}
\alpha ,\beta +1,\gamma +1 \\ 
\delta +1,\varepsilon +1%
\end{array}%
\right\vert 1\right) \\
&=&\frac{\alpha \left( \delta -\beta \right) \left( \delta -\gamma \right) }{%
\delta \left( \delta +1\right) }\,_{3}F_{2}\left( \left. 
\begin{array}{c}
\alpha +1,\beta +1,\gamma +1 \\ 
\delta +2,\varepsilon +1%
\end{array}%
\right\vert 1\right) ,
\end{eqnarray*}%
thus, setting $\alpha =a-1$, $\beta =b-1$, $\gamma =c$, $\delta =2c-b+k$,
and $\varepsilon =c-1$, we arrive at the recursive equation (\ref%
{Recursive_Bailey}). Also, according to \cite[Eqn. 6.4(2)]{Bailey}, after
renaming the parameters, we have%
\begin{equation}
G_{0}\left( a,b,c\right) =\left( 1-\frac{a}{2c}\right) \frac{\Gamma \left(
2c-b+1\right) \Gamma \left( 2c-a-2b\right) }{\Gamma \left( 2c-2b\right)
\Gamma \left( 2c-a-b+1\right) },  \label{G0_Bailey}
\end{equation}%
hence, recursive substitution of (\ref{G0_Bailey})\ in (\ref%
{Recursive_Bailey}), after simplification,\ yields%
\begin{equation*}
G_{1}\left( a,b,c\right) =\frac{\left[ \left( a-2c\right) \left( b-c\right)
+c\right] \Gamma \left( 2c-b+2\right) \Gamma \left( 2c-a-2b+1\right) }{c\
\Gamma \left( 2c-2b+2\right) \Gamma \left( 2c-a-b+2\right) },
\end{equation*}%
and%
\begin{equation*}
G_{2}\left( a,b,c\right) =\frac{\left[ \left( a-2c\right) \left( b-c\right)
+2c\right] \Gamma \left( 2c-b+3\right) \Gamma \left( 2c-a-2b+2\right) }{c\
\Gamma \left( 2c-2b+3\right) \Gamma \left( 2c-a-b+3\right) }.
\end{equation*}%
Therefore, we may conjecture the general form given in (\ref%
{Gk_Bailey_resultado}), which can be proved easily by induction using the
recursive equation (\ref{Recursive_Bailey}).
\end{proof}

\section{Conclusions\label{Section: Conclusions}}

We have obtained some recursive formulas to extend some known $_{2}F_{1}$
and $_{3}F_{2}$ summation formulas by using contiguous relations. These
recursive equations are quite suitable for symbolic and numerical evaluation
by means of computer algebra. Moreover, in some cases, namely (\ref%
{Identity_2F1}), (\ref{Miller_relation}), (\ref{Identity_Pfaff}), and (\ref%
{Gk_Bailey_resultado}), we have derived closed-form expressions. Also, as
by-product, we have obtained an interesting identity in (\ref{Identity_Choi}%
). It is expected that the method used to obtain the different recursive
equations can be applied to extend other hypergeometric summation formulas
given in the literature.

\end{document}